\documentclass[12pt,a4paper]{article}

\usepackage[latin2]{inputenc}
\usepackage{amssymb}
\usepackage{paralist}
\usepackage{amsmath, calc}
\usepackage{amsfonts}
\usepackage{amsthm}
\usepackage{graphicx}
\usepackage{fullpage}
\usepackage{enumerate}
\newtheorem{thm}{Theorem}

\newtheorem{prob}{Problem}
\newtheorem{conj}{Conjecture}
\newtheorem{cor}{Corollary}
\newtheorem{lem}{Lemma}
\newtheorem{rem}{Remark}

\begin{document}

\title{Generalized Sidon sets of perfect powers}
\author{S\'andor Z. Kiss \thanks{Institute of Mathematics, Budapest
University of Technology and Economics, H-1529 B.O. Box, Hungary;
kisspest@cs.elte.hu;
This author was supported by the National Research, Development and Innovation Office NKFIH Grant No. K115288 and K129335. 
This paper was supported by the J\'anos Bolyai Research Scholarship of the Hungarian Academy of Sciences. Supported by the \'UNKP-18-4 New National Excellence Program of the Ministry of 
Human Capacities. Supported by the \'UNKP-19-4 New National Excellence Program 
of the Ministry for Innovation and Technology.}, Csaba
S\'andor \thanks{Institute of Mathematics, Budapest University of
Technology and Economics, MTA-BME Lend\"ulet Arithmetic Combinatorics Research Group H-1529 B.O. Box, Hungary, csandor@math.bme.hu.
This author was supported by the NKFIH Grants No. K129335. Research supported by the Lend\"ulet program of the Hungarian Academy of Sciences (MTA).} 
}
\date{}

\maketitle

\begin{abstract}
\noindent  For $h \ge 2$ and an infinite set
of positive integers $A$, let $R_{A,h}(n)$ denote
the number of solutions of the equation 
\[
a_{1} + a_{2} + \dots + a_{h} = n, \hspace*{3mm} a_{1} \in
A, \dots, a_{h} \in A, \hspace*{3mm} a_{1} < 
a_{2} < \dots{} < a_{h}.
\]
In this paper we prove the
existence of a set $A$ formed by perfect powers with almost possible maximal 
density such that $R_{A,h}(n)$ is bounded by using probabilistic methods. 


\end{abstract}
\section{Introduction}

Let $\mathbb{N}$ denote the set of nonnegative integers and let 
$h, k, m \ge 2$ be integers. For an infinite set
of positive integers $A$, let $R_{A,h}(n)$ and $R^{*}_{A,h}(n)$  denote
the number of solutions of the equations 
\[
a_{1} + a_{2} + \dots + a_{h} = n, \hspace*{3mm} a_{1} \in
A, \dots, a_{h} \in A, \hspace*{3mm} a_{1} < 
a_{2} < \dots{} < a_{h},
\]
\[
a_{1} + a_{2} + \dots + a_{h} = n, \hspace*{3mm} a_{1} \in
A, \dots, a_{h} \in A, \hspace*{3mm} a_{1} \le a_{2} \le \dots{} \le a_{h},
\]
respectively. A set of positive integers $A$ is called $B_h[g]$ set if $R^{*}_{A,h}(n) \le g$ for every postive integer $n$.
We say a set $A$ of nonnegative integers is a basis of order $m$ if every 
nonnegative integer can be 
represented as the sum of $m$ terms from $A$ i.e., $R_{A,m}(n) > 0$ 
for every positive integer $n$. 
Throughout the paper we denote the cardinality of a
finite set $A$ by $|A|$ and we put
\[
A(n) = \sum_{\overset{a \in A}{a \le n}}1.
\]
Furthermore, we write $\mathbb{N}^{k} = \{0^{k}, 1^{k}, 2^{k}, \dots{}\}$ and $(\mathbb{Z}^{+})^{k} = \{1^{k}, 2^{k}, 3^{k}, \dots{}\}$. 
The investigation of the existence of a basis
formed by perfect powers is a classical problem in Number Theory. For
instance, the Waring problem asserts that $\mathbb{N}^{k}$ is a basis of 
order $m$ if $m$ is sufficiently large compared to the power $k$. 
A few years ago, the assertion of Waring was sharpened [17] by proving the existence of a sparse basis formed by perfect powers. More precisely,

\begin{thm}[V.H.Vu]
For any fixed $k \ge 2$, there is a constant $m_{0}(k)$ such that if $m > m_{0}(k)$ then there exists a basis $A \subset \mathbb{N}^{k}$ of order $m$
such that $A(x) \ll x^{1/m}\log^{1/m}x$.
\end{thm}
Obvoiusly, if $A$ is a basis of order $m$, then $A(x)^{m} > \binom{A(x)}{m} \ge x + 1$, which yields 
%Since there are $x + 1$ nonnegative integers up to a positive integer $x$, we have 
$A(x) \gg x^{1/m}$.   

It is natural to ask if
there exists a $B_{h}[g]$ set formed by $k$-th powers such that $A(x)$
is as large as possible. 
%\sim x^{\frac{k}{h}-\varepsilon}$. 
Now, we prove that the best possible exponent is
$\min\left\{\frac{1}{k},\frac{1}{h}\right\}$. It is clear that $A(x) \le x^{1/k}$. 
On the other hand, if $A$ is a $B_{h}[g]$ set, then
\[
hgx \ge \sum_{n=1}^{hx}R^{*}_{A,h}(n) \ge \binom{A(x)}{h} \ge 
\left(\frac{(A(x)-(h-1))^{h}}{h!}\right)
\]
and so $A(x) \le \sqrt[h]{hgx\cdot h!}+h-1$, which implies that
\[
A(x)\ll x^{\min\left\{\frac{1}{k},\frac{1}{h}\right\}}. 
\]
Next, we show that in the special case $k = h = 2$ this can be improved. 
According to a well known theorem of Landau [12], the number of positive integers up to a large $x$ that can be written as the sum of two squares is asymptotically $\frac{cx}{\sqrt{\log x}}$, where $c$ is called Landau-Ramanujan 
constant. On the other hand, if $A$ is a $B_{2}[g]$ set formed by squares, then there are at most $\binom{A(x)}{2}$ integers below $2x$ can be written as the sum of two squares. Then we have
\[
\binom{A(x)}{2} \le \sum_{n=1}^{2x}R^{*}_{A,2}(n) \le (c+o(1))\frac{2x}{\sqrt{\log 2x}},
\]   
which gives 
\[
A(x) \ll \frac{\sqrt{x}}{\sqrt[4]{\log x}}.
\]
In view of the above observations, we can formulate the following conjecture. 

\begin{conj}
For every $k \ge 1$, $h \ge 2$, $\varepsilon > 0$ there exists a $B_{h}[g]$ set
$A \subseteq (\mathbb{Z}^{+})^{k}$ such that
\[
A(x) \gg x^{\min\left\{\frac{1}{k},\frac{1}{h}\right\}-\varepsilon}.
\]
\end{conj}
\noindent The above conjecture was proved by Erd\H{o}s and R\'enyi [5] when
$k = 1$, $h = 2$. It was also proved [3], [17] 
by $k = 1$, $h > 2$.
It is clear that if Conjecture 1 holds for $h = k$, then it holds for
every $2 \le h \le k$ as well.
%Furthermore, consider $h = 2$, $k = 2$. Then, clearly $A(x) \gg x^{1/2}$. 
%It is well known [Hua, Theorem 7.5] that if $d_{1}(n)$ and $d_{3}(n)$ denote
%the number of divisors of $n$ which are congruent $1$ and $3$ modulo $4$ 
%respectively, then $R_{A,2}(n) = 4(d_{1}(n) - d_{3}(n))$, 
%and so $\limsup_{n \rightarrow \infty}R_{A,2}(n) = +\infty$.
Furthermore, it was proved in [2] that for any positive integer $g$ and $\epsilon > 0$, there exists a $B_{2}[g]$ set 
$A$ of squares such that $A(x) \gg x^{\frac{g}{2g+1}-\epsilon} = x^{\frac{1}{2}-\frac{1}{4g+2}-\epsilon}$ 
by using the probabilistic method. This implies Conjecture 1 for $h = k = 2$.
Moreover, a conjecture of Lander, Parkin and Selfridge [13] asserts that 
if the diophantine equation $\sum_{i=1}^{n}x_{i}^{k} = \sum_{j=1}^{m}y_j^{k}$, where
$x_{i} \ne y_{j}$ for all $1 \le i \le n$ and $1 \le j \le m$ has a nontrivial solution, then $n + m \ge k$. If $h < \frac{k}{2}$, this conjecture clearly implies Conjecture 1.
It turns out from Theorem 412 in [10] that
the number of solutions of $a^{3} + b^{3} = c^{3} + d^{3}$ can be made 
arbitrary large, hence the set of cubes is not a $B_{2}[g]$ set for any $g$.  
It is also known [15] that given any real solution of the equation
$a^{4} + b^{4} = c^{4} + d^{4}$, there is a rational solution arbitrary close to 
it, which implies that the quartics cannot be a $B_{2}[1]$. It may happen that they form a $B_{2}[2]$ set. As far as we know, it is not known that the equation $a^{5} + b^{5} = c^{5} + d^{5}$ has any 
nontrivial solution. It is conjectured that the fifth powers form a $B_{2}[1]$ 
set [7,D1]. More generally, Hypothesis K of Hardy and Littlewood [9] asserts 
that if $h = k$, then $R_{(\mathbb{Z}^{+})^{k},k}(n) = O(n^{\varepsilon})$. 
The conjecture is true for $k = 2$ [11, Theorem 7.6] and Mahler proved [14] 
that it is false for $k = 3$. The conjecture is still open for $k \ge 4$ 
[16]. In this paper we prove that if Hypothesis K holds, then there exists a set $A$ of positive perfect powers as dense as in Conjecture 1 such that $R_{A,h}(n)$ is bounded.

\begin{thm}
Let $k$ be a positive integer. Assume that for some $2 \le h \le k$ and for every $\eta > 0$, there exists a positive integer $n_{0}(\eta)$ such that for 
every $n \ge n_{0}(\eta)$, $R_{(\mathbb{Z}^{+})^{k},h}(n) < n^{\eta}$.  
Then for every $\varepsilon > 0$ there exists a set
$A \subseteq (\mathbb{Z}^{+})^{k}$ such that $R_{A,h}(n)$ is bounded and
\[
A(x) \gg x^{\frac{1}{k}-\varepsilon} = x^{\min\{\frac{1}{k},\frac{1}{h}\}-\varepsilon}.
\]
\end{thm}
If $k \ge 2$ is even, then it is clear from [11, Theorem 7.6] that 
$R^{*}_{(\mathbb{Z}^{+})^{k},2}(n) \le R^{*}_{(\mathbb{Z}^{+})^{2},2}(n) = n^{o(1)}$.
If $k \ge 2$ is odd, then clearly 
\[
R^{*}_{(\mathbb{Z}^{+})^{k},2}(n) = |\{(a,b): 1 \le a \le b, a^{k} + b^{k} = n\}|.
\] 
It follows that $a + b$ divides $n$. We show that for every divisor $d$ of $n$,
there is at most one pair $(a,b)$, $1 \le a < b$ such that 
$a + b = d$. For $0 < x < \frac{d}{2}$ consider the function 
$f(x) = x^{k} + (d-x)^{k}$. Since $f(x)$ is continuous and strictly decreasing, 
then it assumes every values at most once. It follows that 
$R^{*}_{(\mathbb{Z}^{+})^{k},2}(n) \le d(n) = n^{o(1)}$, where $d(n)$ is the number of positive divisors of $n$. As a corollary, we get that our conjecture is true for $h = 2$. 

\begin{cor}
For every $k \ge 2$, $\varepsilon > 0$ there exists a $B_{2}[g]$ set
$A \subseteq (\mathbb{Z}^{+})^{k}$ such that
\[
A(x) \gg x^{\frac{1}{k}-\varepsilon} = x^{\min\{\frac{1}{k},\frac{1}{h}\}-\varepsilon}.
\]
\end{cor}
\noindent Furthermore, we do not even know whether there exists $A \subseteq (\mathbb{Z}^{+})^{2}$ such that $R_{A,3}(n)$ is bounded and
$A(x) \gg x^{\frac{1}{3}-\varepsilon}$.

\begin{prob}
Does there exist $A \subseteq (\mathbb{Z}^{+})^{2}$ such that $R_{A,3}(n)$ is bounded and for any $\varepsilon > 0$, $A(x) \gg x^{\frac{1}{3}-\varepsilon}$?
\end{prob}

\begin{thm}
For every $k \ge 2$ there exists a positive integer $h_{0}(k) = O(8^{k}k^{2})$ 
such that for every
$h \ge h_{0}(k)$ and for every $\varepsilon > 0$ there exists a set
$A \subseteq (\mathbb{Z}^{+})^{k}$ such that $R_{A,h}(n)$ is bounded and
\[
A(x) \gg x^{\frac{1}{h}-\varepsilon} = x^{\min\{\frac{1}{k},\frac{1}{h}\}-\varepsilon}.
\]
\end{thm}
\noindent If $f(x)$ and $g(x)$ are real-valued functions, then we denote $f(x) = O(g(x))$
by $f(x) \ll g(x)$. Before we prove Theorem 2 and Theorem 3 we give a short 
survey of the probabilistic method we will use.

\section{Probabilistic and combinatorial tools}
The proofs of Theorem 2 and 3 are based on the probabilistic method due to Erd\H{o}s and R\'enyi. There is an excellent summary of the probabilistic method in the Halberstam - Roth book [8]. First we give
a survey of the probabilistic tools and notations which we use in the proofs of
Theorem 2 and 3. Let $\Omega$ denote the set of the strictly increasing sequences of positive integers. In this paper we denote the probability of an event $E$ by $\mathbb{P}(E)$.

\begin{lem}
Let 
\[
\alpha_{1}, \alpha_{2}, \alpha_{3} \dots{} 
\]
be real numbers satisfying 
\[
0 \le \alpha_{n} \le 1 \hspace*{4mm} (n = 1, 2, \dots{}).
\]
Then there exists a probability space ($\Omega$, $\mathcal{S}$, $\mathbb{P}$) with the
following two properties:
\begin{itemize}
\item[(i)] For every natural number $n$, the event $E^{(n)} = \{A$:
  $A \in \Omega$, $n \in A\}$ is measurable, and
  $\mathbb{P}(E^{(n)}) = \alpha_{n}$.
\item[(ii)] The events $E^{(1)}$, $E^{(2)}$, ... are independent.
\end{itemize}
\end{lem}
See Theorem 13. in [8], p. 142. 
We denote the
characteristic function of the event $E^{(n)}$ by
\[
\varrho(A, n) = 
\left\{
\begin{aligned}
1 \textnormal{, if } n \in A \\
0 \textnormal{, if } n \notin A.
\end{aligned} \hspace*{3mm}
\right.
\]
Thus
\begin{equation}
A(n) = \sum_{j=1}^{n}\varrho(A, j).
\end{equation}

\noindent Furthermore, we denote the number of solutions of
$a_{i_{1}} + a_{i_{2}} + \dots{} + a_{i_{h}} = n$ by $R_{A,h}(n)$, where 
$a_{i_{1}} \in A$, $a_{i_{2}} \in A$, ...,$a_{i_{h}} \in A$, $1 \le a_{i_{1}} < a_{i_{2}} \dots{}
< a_{i_{h}} < n$.  
Thus 
\begin{equation}
R_{A,h}(n) = \sum_{\overset{(a_{1}, a_{2}, \dots{}, a_{h}) \in \mathbb{N}^{h}}{1
    \le a_{1} < \dots{} < a_{h} < n}\atop {a_{1} + a_{2} + \dots{} + a_{h} =
    n}}\varrho(A, a_{1})\varrho(A, a_{2}) \dots{}
\varrho(A, a_{h}).
\end{equation}
%We also need the following important lemma:

%\begin{lem}
%If the sequence (1) satisfies (2) and 
%\[
%\alpha_{j} = {\alpha}j^{-c} \hspace*{3mm} for \hspace*{3mm} j \ge j_{0}, 
%\]
%\noindent where $\alpha$, $c$ are constants such that $0 < \alpha$, $0 < c <
%1$, then with probability 1, we have
%\[
%A(n) \sim \frac{\alpha}{1 - c}n^{1-c}.
%\]
%\end{lem}
%\noindent This lemma is a consequence of Lemmas 10 and 11 in [1], pp. 144 - 145. 
The following lemma is called Borel - Cantelli lemma.

\begin{lem}
Let ($\mathcal{X}$, $\mathcal{S}$, $\mathbb{P}$) be a probability space and let $F_{1}$, $F_{2}$, ... be
a sequence of measurable events. If 
\[
\sum_{j=1}^{+\infty}\mathbb{P}(F_{j}) < +\infty,
\]
\noindent then with probability 1, at most a finite number of the events
$F_{j}$ can occur.
\end{lem}
See [8], p. 135. The next lemma is called the disjointness lemma due to Erd\H{o}s and Tetali.
\begin{lem}
Let $B_{1}$, $B_{2} \dots{}$ be events in a probability space. If $\sum_{i}\mathbb{P}(B_{i}) \le \mu$, then
\[
\sum_{\overset{(B_{i_{1}}, \dots{} ,B_{i_{l}l})}{independent}}\mathbb{P}(B_{i_{1}} \cap
\dots{} \cap B_{i_{l}}) \le {\mu}^{l}/l!.
\]
\end{lem}
\noindent See [6] for the proof. We will use the following special case of Chernoff's inequality (Corollary 1.9. in
 [1]).
\begin{lem}
If $t_i$'s are independent Boolean random
 variables (i.e., every $t_{i} \in \{0,1\}$ 
and $X = t_1 + \dots{} + t_n$, then for any $\delta > 0$ we have 
\[
\mathbb{P}\big(|X - \mathbb{E}(X)| \ge \delta\mathbb{E}(X)\big) \le 2e^{-min(\delta^{2}/4, \delta/2)\mathbb{E}(X)}. 
\]
\end{lem}
Finally we need the following combinatorial result due
to Erd\H{o}s and Rado, see [4]. Let $r$ be a positive integer, $r \ge 3$. A
collection of sets $A_{1}, A_{2}, \dots{} A_{r}$ forms a $\Delta$ - system if
the sets have pairwise the same intersection.  

\begin{lem}
If $H$ is a collection of sets of size at most $k$ and $|H| > (r - 1)^{k}k!$
then $H$ contains $r$ sets forming a $\Delta$ - system.
\end{lem}

\section{Proof of Theorem 2}

In the first step, we prove that for any random set $A$, if the expectation of 
$R_{A,h}(n)$ is small, then it is almost always bounded.

\begin{lem}
Let $2 \le h \le k$. Consider a random set $A \subset \mathbb{N}$ defined by 
$\alpha_{n} = \mathbb{P}(n\in A)$. If $\mathbb{E}(R_{A,l}(n)) \ll n^{-\varepsilon}$
for every $\varepsilon > 0$ and for every $2 \le l \le h$, then $R_{A,h}(n)$ is bounded with probability $1$.
\end{lem}

\begin{proof}
We show similarly as in [6] that with probability $1$, $R_{A,h}(n)$ is bounded by a 
constant. For each representation $a_{1} + \dots{} + a_{h} = n$, $a_{1} < \dots{} < a_{h}$, $a_{1}, \dots{} ,a_{h} \in A$ we assign a set 
$S = \{a_{1}, \dots{} ,a_{h}\}$. We say two representations 
$a_{1} + \dots{} + a_{h} = b_{1} + \dots{} + b_{h} = n$ are disjoint if the assigned sets $S_{1} = \{a_{1}, \dots{} ,a_{h}\}$ and $S_{2} = \{b_{1}, \dots{} ,b_{h}\}$ 
are disjoint.

For $2 \le l \le h$ and a set of positive integers $B$, 
let $f_{B,l}(n) = f_{l}(n)$ denote the  
maximum number of pairwise disjoint representations of $n$ as the sum of $l$ distinct terms from $B$. Let 
\[
\mathcal{B} = \{(a_{1}, \dots{}, a_{l}): a_{1} + \dots{} + a_{l} = n, a_{1}
\in A, \dots{}, a_{l} \in A, 1 \le a_{1} < \dots{} < a_{l} < n\},
\]
\noindent and let $H(\mathcal{B}) = \{\mathcal{T} \subset \mathcal{B}$: 
all the $S \in \mathcal{T}$ are pairwise disjoint$\}$. 
\noindent It is clear that the pairwise disjointness of the sets implies the 
independence of the associated events, i.e., if $S_1$ and $S_2$ are pairwise
disjoint representations, the events $S_1 \subset A$, $S_2 \subset A$ 
are independent. On the other hand, for a fixed $2 \le l \le h$, let $E_{n}$ denote the event
\[
E_{n} = \{A: A \in \Omega, f_{A,l}(n) > G\}
\]
\noindent for some $G$ and write 
\[
\mathcal{F} = \Omega \setminus
\bigcap_{i=1}^{+\infty}\Big(\bigcup_{n=i}^{+\infty}E_{n}\Big). 
\]
\noindent As a result, we see that $A \in \mathcal{F}$ if and only if there exists a number $n_{1} = n_{1}(A)$ such that we have 
\[
f_{A,l}(n) \le G \hspace*{4mm} for \hspace*{4mm} n \ge n_{1}.
\]
\noindent We will prove that $\mathbb{P}(\mathcal{F}) = 1$ if $G$ is large 
enough. Directly from Lemma 3, for $G = \Big[\frac{1}{\varepsilon}\Big]$ with any $\varepsilon > 0$, we have 
\[
\mathbb{P}(f_{l}(n) > G) \le \mathbb{P}\Big(\bigcup_{\overset{\mathcal{T} \subset
    H(\mathcal{B})}{|\mathcal{T}| = G+1}}\bigcap_{S \in \mathcal{T}}S\Big) \le
\sum_{\overset{\mathcal{T} \subset H(\mathcal{B})}{|\mathcal{T}| = G+1}}\mathbb{P}\Big(\bigcap_{S \in \mathcal{T}}S\Big)
\]
\[
= \sum_{\overset{(S_{1}, \dots{} ,S_{G+1})}{Pairwise \atop disjoint}}\mathbb{P}(S_{1}
\cap \dots{} \cap S_{G+1}) \le \frac{1}{(G+1)!}(\mathbb{E}(f_{l}(n)))^{G+1} \le
\frac{1}{(G+1)!}(\mathbb{E}(R_{A,l}(n)))^{G+1} 
\]
\[
\ll n^{-(G+1)\varepsilon} \ll n^{-1-\varepsilon}.
\]
Using the Borel - Cantelli lemma, it follows that with probability 1, for  $2 \le l \le h$ there exists  an $n_{l}$ such that 
\[
f_{l}(n) \le G \hspace*{2mm} for \hspace*{2mm} n > n_{l}.
\]
On the other hand, for any finite $n_{l}$, there are at most a finite number
of representations as a sum of $l$ terms. Therefore, almost always
for $2 \le l \le h$ there exists a $c_{l}$ such that for every $n$,
$f_{l}(n) < c_{l}$. Set $c_{max} = max_{l}\{c_{l}\}$. 
Now we show similarly as in [6] that almost always there exists a constant $c
= c(A)$ such that for every $n$, $R_{A,h}(n) < c$.
Suppose that for some positive integer $m$, 
\[
R_{A,h}(m) > (c_{max})^{h}h! 
\]
with positive probability. Let $H$ be the set of
representations of $m$ as the sum of $h$ distinct terms from
$A$. Then $|H| = R_{A,h}(m) > (c_{max})^{h}h$, thus by Lemma 5,
$H$ contains a $\Delta$ - system $\{S_{1}, \dots{} ,S_{c_{max} +1}\}$.    
If $S_{1} \cap \dots{} \cap S_{c_{max} +1} = \emptyset$, then $S_{1}, \dots{} ,S_{c_{max} +1}$ form a family of disjoint representations of $m$ as the sum of $h$ terms, which contradicts the definition of $c_{max}$. 
Otherwise let the $S_{1} \cap \dots{} \cap S_{c_{max} +1}  = \{x_{1}, \dots{} ,x_{r}\} = S$, where $0 < r < h - 1$. If $\sum_{i=1}^{r}x_{i} = t$,
then $S_{1}\setminus S, \dots{} ,S_{c_{max} +1}\setminus S$ form a family of disjoint representations of $m - t$ as the sum of $h - r$ terms. It follows that
$f_{h-r}(m-t) \ge c_{max} + 1 > c_{h-r}$, which is impossible because of the definition of $c_{max}$. As a result, we see that $R_{A,h}(m) \le (c_{max})^{h}h!$, which implies that $R_{A,h}(n)$ is bounded with probability 1.
\end{proof}

\begin{rem}
It follows from the proof of Lemma 6 that the representation function $R_{A,h}(n)$ is bounded if and only if $f_{l}(n)$ is bounded for every $2 \le l \le h$. 
\end{rem}

\begin{lem}
Consider a random set $A \subset \mathbb{N}$ defined by 
$\alpha_{n} = \mathbb{P}(n\in A)$. If 
\[
\lim_{x \rightarrow \infty}\frac{\mathbb{E}(A(x))}{\log x} = +\infty, 
\]
then $A(x) \sim \mathbb{E}(A(x))$, with probability $1$.
\end{lem}

\begin{proof}
It is clear from (1) that $A(x)$ is the sum of independent Boolean random
 variables. Note that $\delta < 2$ if $x$ is large enough, thus it follows from Lemma 4 with 
\[
\delta = \sqrt{\frac{8\log x}{\mathbb{E}(A(x))}}
\]
that  
\[
\mathbb{P}\left(|A(x) - \mathbb{E}(A(x))| \ge \sqrt{\frac{8\log x}{\mathbb{E}(A(x))}}\cdot \mathbb{E}(A(x))\right) \le 2\exp\left(-\frac{1}{4}\cdot \frac{8\log x}{\mathbb{E}(A(x))}\cdot \mathbb{E}(A(x)) \right) = \frac{2}{x^{2}}.   
\]
Since $\sum_{x=1}^{\infty}\frac{2}{x^{2}}$ converges, 
by the Borel - Cantelli lemma we have
\[
|A(x) - \mathbb{E}(A(x))| < \frac{8\log x}{\mathbb{E}(A(x))}\cdot \mathbb{E}(A(x))
\]
with probability $1$, for every $x$ large enough. Since
\[
\sqrt{\frac{8\log x}{\mathbb{E}(A(x))}} = o(1),
\]
the statement of Lemma 7 follows.
\end{proof}

Now we are ready to prove Theorem 2. In the first step, we show that 
%for every $\eta > 0$, there exists an $n_{0}(\eta)$ such that 
%\[
%R_{(\mathbb{Z}^{+})^{k},h}(n) < n^{\eta}
%\]
%for every $n > n_{0}(\eta)$, then 
for every $2 \le l \le h$ and $0 < \kappa < \frac{1}{k}$ there exists an $n_{0}(\kappa,l)$ such that 
\[
R_{(\mathbb{Z}^{+})^{k},l}(n) < n^{\kappa}
\]
for every $n > n_{0}(\kappa,l)$. We prove by contradiction. Suppose that there exist a constant $c > 0$ and a $0 < \kappa < \frac{1}{k}$ such that
\[
R_{(\mathbb{Z}^{+})^{k},l}(n) > cn^{\kappa}
\]
for infinitely many $n$. Pick a large $n$ and consider different representations $n = a^{k}_{1,1} + a^{k}_{1,2} + \dots{} + a^{k}_{1,l}, n = a^{k}_{2,1} + a^{k}_{2,2} + \dots{} + a^{k}_{2,l}, \dots{} 
,n = a^{k}_{u,1} + a^{k}_{u,2} + \dots{} + a^{k}_{u,l}$ where $a_{i,1} < a_{i,2} < \dots{} < a_{i,l}$ positive integers for every $1 \le i \le u$, where $cn^{\kappa} < u < \frac{1}{2}n^{1/k}$. 
Then there exist $1 \le b_{1} < b_{2} < \dots{} < b_{h-l} \le n^{1/k}$ positive integers such that $b_{v} \ne a_{i,j}$ for every $1 \le v \le h-l$ and $1 \le i \le u$, $1 \le j \le l$. 
Then we have
\[
R_{(\mathbb{Z}^{+})^{k},h}(n+b^{k}_{1} + \dots{} + b^{k}_{h-l}) \ge cn^{\kappa} = \frac{c}{h^{\kappa}}(nh)^{\kappa} \ge \frac{c}{h^{\kappa}}(n+b^{k}_{1} + \dots{} + b^{k}_{h-l})^{\kappa}.
\]
If we denote $m = n + b^{k}_{1} + \dots{} + b^{k}_{h-l}$, then $R_{(\mathbb{Z}^{+})^{k},h}(m) > \frac{c}{h^{\kappa}}m^{\kappa}$ for infinitely many $m$. It follows that there exist infinitely many $m$ such that
\[
R_{(\mathbb{Z}^{+})^{k},h}(m) > m^{\kappa/2}.
\]
In view of the hypothesis in Theorem 2 we get a contradisction. 

\indent Next, for an $\varepsilon > 0$, we define the random set $A$ by 
\[
\mathbb{P}(n \in A) = 
\left\{
\begin{aligned}
\frac{1}{n^{\varepsilon}} \textnormal{, if } n \in (\mathbb{Z}^{+})^{k}\\ 
0 \textnormal{, if } n \notin (\mathbb{Z}^{+})^{k}.
\end{aligned} \hspace*{3mm}
\right.
\]
Then, in view of (2) and $x_{l} \ge \frac{n}{l}$, for every $2 \le l \le h$ we 
have
\[
\mathbb{E}(R_{A,l}(n)) = \sum_{\overset{(x_{1}, \dots{} ,x_{l})\in (\mathbb{Z}^{+})^{k}}{1 \le x_{1} < \dots{} < x_{l}} 
\atop {x_{1} + \dots{} + x_{l} = n}}\mathbb{P}(x_{1}, \dots{}, x_{l} \in A) = \sum_{\overset{(x_{1}, \dots{} ,x_{l})\in (\mathbb{Z}^{+})^{k}}{1 \le x_{1} < \dots{} < x_{l}} 
\atop {x_{1} + \dots{} + x_{l} = n}}\frac{1}{(x_{1}\dots{} x_{l})^{\varepsilon}} 
\]
\[
\ll \frac{1}{n^{\varepsilon}}\cdot R_{(\mathbb{Z}^{+})^{k},l}(n) \ll \frac{1}{n^{\varepsilon}}\cdot n^{\varepsilon/2} = n^{-\varepsilon/2}.
\]
It follows from Lemma 6 that $R_{A,h}(n)$ is almost always bounded. 
In the next step, we prove that $A$ is as dense as desired.
Applying the Euler-Maclaurin integral formula,
\[
\mathbb{E}(A(x)) = \sum_{m \le x^{1/k}}\frac{1}{(m^{k})^{\varepsilon}} = \int_{0}^{x^{1/k}}t^{-k\varepsilon}dt + O(1) = \frac{1}{1-k\varepsilon}x^{\frac{1}{k}-\varepsilon} + O(1).
\] 
By Lemma 7, assuming $\varepsilon < \frac{1}{k}$ we get that
\[
A(x) \gg x^{\frac{1}{k}-\varepsilon}
\]
with probability 1. The proof of Theorem 2 is completed.

\section{Proof of Theorem 3}

As $h > k$, we define the random set $A$ by
\[
\mathbb{P}(n \in A) = 
\left\{
\begin{aligned}
\frac{1}{n^{\frac{1}{k}-\frac{1}{h}+\varepsilon}} \textnormal{, if } n \in (\mathbb{Z}^{+})^{k}\\ 
0 \textnormal{, if } n \notin (\mathbb{Z}^{+})^{k}.
\end{aligned} \hspace*{3mm}
\right.
\]
\noindent First, for every $2 \le l \le h$, we give an upper estimation to $\mathbb{E}(R_{A,l}(n))$, where
\begin{equation}
\mathbb{E}(R_{A,l}(n)) = \sum_{\overset{(x_{1}, \dots{} ,x_{l})\in (\mathbb{Z}^{+})^{k}}{1 \le x_{1} < \dots{} < x_{l}} 
\atop {x_{1} + \dots{} + x_{l} = n}}\mathbb{P}(x_{1}, \dots{}, x_{l} \in A).
\end{equation}
We prove that there exists $h_{1}(k)$ such that for $h \ge h_{1}(k)$, $l \le h$, we have
\[
\sum_{\overset{(x_{1}, \dots{} ,x_{l})\in (\mathbb{Z}^{+})^{k}}{1 \le x_{1} < \dots{} < x_{l}} 
\atop {x_{1} + \dots{} + x_{l} = n}}\mathbb{P}(x_{1}, \dots{}, x_{l} \in A) \ll n^{-\varepsilon}.
\] 
Assume that $l \le \frac{h}{k}$. Then we have
\[
\sum_{\overset{(x_{1}, \dots{} ,x_{l})\in (\mathbb{Z}^{+})^{k}}{1 \le x_{1} < \dots{} < x_{l}} \atop{x_{1} + \dots{} + x_{l} = n}}\mathbb{P}(x_{1}, \dots{}, x_{l} \in A) = \sum_{\overset{(x_{1}, \dots{} ,x_{l})\in (\mathbb{Z}^{+})^{k}}{1 \le x_{1} < \dots{} < x_{l}}\atop {x_{1} + \dots{} + x_{l} = n}}\frac{1}{(x_{1} \cdots{} x_{l})^{\frac{1}{k}-\frac{1}{h}+\varepsilon}}.
\]
Since $x_{l} \ge \frac{n}{l}$, we may therefore calculate
\[
\sum_{\overset{(x_{1}, \dots{} ,x_{l})\in (\mathbb{Z}^{+})^{k}}{1 \le x_{1} < \dots{} < x_{l}}\atop {x_{1} + \dots{} + x_{l} = n}}\frac{1}{(x_{1} \cdots{} x_{l})^{\frac{1}{k}-\frac{1}{h}+\varepsilon}} \ll
n^{-\left(\frac{1}{k}-\frac{1}{h}+\varepsilon\right)}\left(\sum_{x=1}^{n^{1/k}}
\frac{1}{x^{k(\frac{1}{k}-\frac{1}{h}+\varepsilon)}}\right)^{l-1} 
\]
\[
= n^{-\frac{1}{k}+\frac{1}{h}-\varepsilon}\left(\sum_{x=1}^{n^{1/k}}\frac{1}{x^{1-\frac{k}{h}+k\varepsilon}}\right)^{l-1}. 
\]
Then, on applying the assumption $\frac{l}{h} \le \frac{1}{k}$, we find via Euler-Maclaurin integral formula that 
\[
\sum_{\overset{(x_{1}, \dots{} ,x_{l})\in (\mathbb{Z}^{+})^{k}}{1 \le x_{1} < \dots{} < x_{l}}\atop {x_{1} + \dots{} + x_{l} = n}}\frac{1}{(x_{1} \cdots{} x_{l})^{\frac{1}{k}-\frac{1}{h}+\varepsilon}}
\ll n^{-\frac{1}{k}+\frac{1}{h}-\varepsilon+(l-1)(\frac{k}{h}-k\varepsilon)\frac{1}{k}} = n^{-\frac{1}{k}+\frac{1}{h}-\varepsilon+(l-1)(\frac{1}{h}-\varepsilon)} 
\]
\[
= n^{-\frac{1}{k}+\frac{1}{h}-\varepsilon+\frac{l}{h}-\frac{1}{h}+\varepsilon-l\varepsilon} 
\ll n^{-l\varepsilon} \ll n^{-\varepsilon}.  
\]
Now we assume that $\frac{h}{k} < l \le h$. Then we have 
\[
\sum_{\overset{(x_{1}, \dots{} ,x_{l})\in (\mathbb{Z}^{+})^{k}}{1 \le < x_{1} < \dots{} < x_{l}}\atop {x_{1} + \dots{} + x_{l} = n}}\mathbb{P}(x_{1}, \dots{}, x_{l} \in A) 
= \sum_{\overset{(x_{1}, \dots{} ,x_{l})\in (\mathbb{Z}^{+})^{k}}{1 \le < x_{1} < \dots{} < x_{l}}\atop {x_{1} + \dots{} + x_{l} = n}}\frac{1}{(x_{1} \cdots{} x_{l})^{\frac{1}{k}-\frac{1}{h}+\varepsilon}}
\]
\[
\le \sum_{\overset{(x_{1}, \dots{} ,x_{l})\in (\mathbb{Z}^{+})^{k}}{1 \le < x_{1} < \dots{} < x_{l}}\atop {x_{1} + \dots{} + x_{l} = n}}\frac{1}{(x_{1} \cdots{} x_{l})^{\frac{1}{k}-\frac{1}{l}+\varepsilon}}
= \sum_{\overset{(y_{1}, \dots{} ,y_{l})\in (\mathbb{Z}^{+})^{l}}{1 \le < y_{1} < \dots{} < y_{l}}\atop {y_{1}^{k} + \dots{} + y_{l}^{k} = n}}\frac{1}{(y_{1} \cdots{} y_{l})^{1-\frac{k}{l}+k\varepsilon}}.
\]
We need the following lemma, which is a weaker version of a lemma of Vu [17, Lemma 2.1].

\begin{lem} For a fixed $k \ge 2$ there exists a constant 
$h_{2}(k) = O(8^{k}k^{2})$ such that for any $l \ge h_{2}(k)$ and for every $P_{1}, \dots{}, P_{l} \in \mathbb{Z}^{+}$ we have
\[
|\{(y_{1}, \dots{} ,y_{l}):  y_{i} \in \mathbb{Z}^{+}, y_{i} \le P_{i}, y_{1}^{k} + \dots{} + y_{l}^{k} = n\}| \ll \frac{1}{n}\prod_{i=1}^{l}P_{i} + \left(\prod_{i=1}^{l}P_{i}\right)^{1-\frac{k}{l}}.
\]
\end{lem}
By Lemma 8 one has 
\[
\sum_{\overset{(y_{1}, \dots{} ,y_{l})\in (\mathbb{Z}^{+})^{l}}{\frac{P_{i}}{2} < y_{i} \le P_{i}}\atop {y_{1}^{k} + \dots{} + y_{l}^{k} = n}}\frac{1}{(y_{1} \cdots{} y_{l})^{1-\frac{k}{l}+k\varepsilon}} \ll
\left(\frac{1}{n}\prod_{i=1}^{l}P_{i} + \left(\prod_{i=1}^{l}P_{i}\right)^{1-\frac{k}{l}}\right)\left(\prod_{i=1}^{l}P_{i}\right)^{-1+\frac{k}{l}-k\varepsilon}
\]
\[
\ll \frac{1}{n}\left(\prod_{i=1}^{l}P_{i}\right)^{\frac{k}{l}-k\varepsilon} + \left(\prod_{i=1}^{l}P_{i}\right)^{-k\varepsilon}. 
\]
Let $(P_{1}, \dots{} ,P_{l}) = (2^{i_{1}}, \dots{} ,2^{i_{l}})$, where $0 \le i_{1} \le i_{2} \le \dots{} \le i_{l}$. If $1 \le y_{1} < y_{2} < \dots{} < y_{l}$, $\sum_{i=1}^{l}y_{i}^{k} = n$, then
obviously
\[
\left(\frac{n}{l}\right)^{1/k} \le y_{l} \le n^{1/k}.
\]
Consequently
\begin{equation} 
i_{l} \le \frac{1}{k}\log_{2}n + 1.
\end{equation}
Then, by Lemma 8 we have 
\[
Q = \sum_{\overset{(i_{1}, \dots{} ,i_{l})}{0 \le i_{1} \le \dots{} \le i_{l}}}
\sum_{\overset{(y_{1}, \dots{} ,y_{l})}{1 \le < y_{1} < \dots{} < y_{l}}\atop{{2^{i_{j}-1} < y_{j} \le 2^{i_{j}}}\atop {y_{1}^{k} + \dots{} + y_{l}^{k} = n}}}\frac{1}{(y_{1} \cdots{} y_{l})^{1-\frac{k}{l}+k\varepsilon}} 
\ll \sum_{\overset{(i_{1}, \dots{} ,i_{l})}{0 \le i_{1} \le \dots{} \le i_{l}}}\frac{1}{n}\left(\prod_{i=1}^{l}2^{i_{j}}\right)^{\frac{k}{l}-k\varepsilon} + \sum_{\overset{(i_{1}, \dots{} ,i_{l})}{0 \le i_{1} \le \dots{} \le i_{l}}}\left(\prod_{i=1}^{l}2^{i_{j}}\right)^{-k\varepsilon} 
\]
\[
= Q_{1} + Q_{2}.
\]
In the first step, we estimate $Q_{1}$. By using (4), we have
\[
Q_{1} = \sum_{\overset{(i_{1}, \dots{} ,i_{l})}{0 \le i_{1} \le \dots{} \le i_{l}}}\frac{1}{n}\left(\prod_{i=1}^{l}2^{i_{j}}\right)^{\frac{k}{l}-k\varepsilon} \le \frac{1}{n}\left(\prod_{i=1}^{l}\left(\sum_{i_{j}=1}^{\lfloor \frac{1}{k}\log_{2}n + 1 \rfloor}2^{i_{j}}\right)\right)^{\frac{k}{l}-k\varepsilon}
\]
\[
= \frac{1}{n}\left(\prod_{i=1}^{l}\left(2^{\lfloor \frac{1}{k}\log_{2}n + 1 \rfloor} - 2\right)\right)^{\frac{k}{l}-k\varepsilon} = \frac{1}{n}\left(\left(2^{\lfloor \frac{1}{k}\log_{2}n + 1 \rfloor} - 2\right)^{l}\right)^{\frac{k}{l}-k\varepsilon}
\]
\[
\ll \frac{1}{n}((n^{1/k})^{l})^{\frac{k}{l}-k\varepsilon} = n^{-l\varepsilon} \ll n^{-\varepsilon}.
\]
Next, we estimate $Q_{2}$. Using also (4) we get that
\[
Q_{2} = \sum_{\overset{(i_{1}, \dots{} ,i_{l})}{0 \le i_{1} \le \dots{} \le i_{l}}}\left(\prod_{i=1}^{l}2^{i_{j}}\right)^{-k\varepsilon} \le \left(\prod_{i=1}^{l}\left(\sum_{i_{j}=1}^{\lfloor \frac{1}{k}\log_{2}n + 1 \rfloor}2^{i_{j}}\right)\right)^{-k\varepsilon}
\]
\[
= \left(\prod_{i=1}^{l}\left(2^{\lfloor \frac{1}{k}\log_{2}n + 1 \rfloor} - 2\right)\right)^{-k\varepsilon} = \left(\left(2^{\lfloor \frac{1}{k}\log_{2}n + 1 \rfloor} - 2\right)^{l}\right)^{-k\varepsilon}
\]
\[
\ll ((n^{l/k})^{l})^{-k\varepsilon} = n^{-l\varepsilon} \ll n^{-\varepsilon}.
\]
Grouping these estimates together, 
\[
Q = Q_{1} + Q_{2} \ll n^{-\varepsilon}. 
\]
Returning to (3) we now have the estimation
\[
\mathbb{E}(R_{A,l}(n)) \ll n^{-\varepsilon}.
\]
It follows from Lemma 6 that, with probability 1, $R_{A,h}(n)$ is bounded. 
On the other hand, by using the Euler-Maclaurin formula,
\[
\mathbb{E}(A(x)) = \sum_{m \le x^{1/k}}\frac{1}{(m^{k})^{\frac{1}{k}-\frac{1}{h}+\varepsilon}} = \int_{0}^{x^{1/k}}t^{-1+\frac{k}{h}-k\varepsilon}dt + O(1) = \frac{1}{\frac{k}{h}-k\varepsilon}x^{\frac{1}{h}-\varepsilon} + O(1),
\]
which implies that $A(x) \gg x^{\frac{1}{h}-\varepsilon}$ with probability 1.

\bigskip

\noindent \textit{Remark.} One might like to generalize Theorem 2 and Theorem 3 to $B_{h}[g]$ sets, i.e., to prove the existence of a set $A$ formed by perfect powers 
such that $R^{*}_{A,h}(n) \le g$ for some $g$ and $A$ is as dense as possible. 
To do this, one needs a generalization of Lemma 6 and Lemma 8 for the number of representations of $n$ as linear forms like 
$b_{1}x_{1} + \dots{} + b_{s}x_{s} = n$. Lemma 6 can be extended to linear forms but the generalization of Lemma 8 seems more complicated.

\end{document}